\documentclass{amsart}
\usepackage{amsmath, amscd, amssymb, amsthm}
\usepackage{bbm}
\usepackage{latexsym}
\usepackage{amsfonts}
\usepackage{graphicx}
\usepackage[all,cmtip]{xy}
\usepackage[colorlinks,linkcolor=blue,breaklinks=blue,urlcolor=blue,citecolor=blue,anchorcolor=blue,pagebackref]{hyperref}%
\usepackage{geometry}
\newtheorem{theorem}{Theorem}
\newtheorem{lemma}{Lemma}
\newtheorem{corollary}[theorem]{Corollary}

\newtheorem{proposition}{Proposition}

\newtheorem{conjecture}{Conjecture}

\renewcommand*\backref[1]{}
\renewcommand*\backrefalt[4]{ \ifcase #1 \or (cited on page #2) \else (cited on pages #2) \fi}

\newcommand{\be}{\begin{equation}}
\newcommand{\ee}{\end{equation}}
\newcommand{\bea}{\begin{eqnarray}}
\newcommand{\eea}{\end{eqnarray}}

\newcommand{\vs}{\vspace{0.5cm}}

\def\XXint#1#2#3{{\setbox0=\hbox{$#1{#2#3}{\int}$ }
\vcenter{\hbox{$#2#3$ }}\kern-.6\wd0}}

\begin{document}

\title[On Strominger space forms]{On Strominger space forms}

\author{Shuwen Chen}
\address{Shuwen Chen. School of Mathematical Sciences, Chongqing Normal University, Chongqing 401331, China}
\email{{3153017458@qq.com}}\thanks{Zheng is partially supported by National Natural Science Foundations of China
with the grant No.12071050 and  12141101, Chongqing grant cstc2021ycjh-bgzxm0139, and is supported by the 111 Project D21024.}

\author{Fangyang Zheng}
\address{Fangyang Zheng. School of Mathematical Sciences, Chongqing Normal University, Chongqing 401331, China}
\email{20190045@cqnu.edu.cn; franciszheng@yahoo.com} \thanks{}


\begin{abstract}
In this article, we propose the following conjecture: if the Strominger connection of a compact Hermitian manifold  has constant non-zero holomorphic sectional curvature, then the Hermitian metric must be K\"ahler. The main result of this article is to confirm the conjecture in dimension $2$. We also verify the conjecture in higher dimensions in a couple of special situations.
\end{abstract}

\maketitle

\tableofcontents

\markleft{Shuwen Chen and Fangyang Zheng}
\markright{Strominger space forms}

\section{Introduction and statement of results}\label{intro}

In Riemannian geometry, the simplest Riemannian manifolds are the {\em space forms,} namely, complete Riemannian manifolds with constant sectional curvature. They are respectively quotients of the sphere $S^n$, the Euclidean space ${\mathbb R}^n$, and the hyperbolic space ${\mathbb H}^n$, equipped with (scaling of) the standard metrics.

In the complex case, one could no longer require the sectional curvature to be constant (unless it is flat). Instead, one requires the {\em holomorphic sectional curvature} to be constant. Complete K\"ahler manifolds with constant holomorphic sectional curvature are called {\em complex space forms}. They are quotients of complex projective space ${\mathbb C}{\mathbb P}^n$, the complex Euclidean space ${\mathbb C}^n$, and the complex hyperbolic space ${\mathbb C}{\mathbb H}^n$, equipped with (scaling of) the standard metrics.

Note that by the classic Schur's Lemma, if the sectional curvature of a Riemannian  manifold of dimension $\geq 3$ is pointwise constant, then it must be constant. Similarly, if the holomorphic sectional curvature of a K\"ahler manifold of complex dimension$\geq 2$ is pointwise constant, then it must be constant. So in the above definition of space forms or complex space forms, one could require the curvature to be {\em pointwise constant} instead.

Given a Hermitian manifold $(M^n,g)$, there are three canonical connections associated with $g$ that are widely studied: the Levi-Civita (Riemannian) connection $\nabla$, the Chern connection $\nabla^c$, and the Strominger (also known as Bismut) connection $\nabla^s$. The last connection was considered by Strominger \cite{Strominger} in 1986 and independently by Bismut \cite{Bismut} in 1989. When $g$ is K\"ahler, these three connections coincide, but when $g$ is not K\"ahler, they are mutually distinct.

For studies on these connections and non-K\"ahler Hermitian geometry in general, we refer the readers to \cite{AI}, \cite{AOUV}, \cite{AU}, \cite{Belgun}, \cite{Belgun1}, \cite{EFV}, \cite{FT}, \cite{FinoTomassini}, \cite{FinoVezzoni}, \cite{Fu}, \cite{FuYau}, \cite{FuLiYau}, \cite{Fu-Zhou}, \cite{GHR}, \cite{Gauduchon}, \cite{Gauduchon1}, \cite{Gray}, \cite{IP}, \cite{Zheng}, and the references therein.

Let us denote by $R$, $R^c$, $R^s$ the curvature tensor of the connections $\nabla$, $\nabla^c$, $\nabla^s$, respectively. The holomorphic sectional curvature $H$ of $\nabla$ is defined by
$$ H(X) = \frac{R_{X\overline{X} X\overline{X} } }{|X|^4}$$
where $X$ is any non-zero complex tangent vector of type $(1,0)$.  $H^c$ and $H^s$ can be similarly defined, replacing $R$ by $R^c$ or $R^s$, respectively. For each of these three canonical connections, a natural question is to understand the corresponding (weak) space forms, namely, those Hermitian manifolds with (pointwise) constant holomorphic sectional curvature $H$ (or $H^c$, or $H^s$).

Note that in the non-K\"ahler case, we no longer have Schur's Lemma, so the pointwise constancy of holomorphic sectional curvature does not imply constancy in general, at least for non-compact manifolds. Nonetheless, a long term folklore conjecture in non-K\"ahler geometry is the following:

\begin{conjecture} \label{conj1}
Let $(M^n,g)$ be a compact Hermitian manifold with $n\geq 2$. Assume that $H=c$ (or $H^c=c$) where $c$ is a constant. If $c\neq 0$, then $g$ is K\"ahler, and if $c=0$, then $R=0$ (or $R^c=0$).
\end{conjecture}

Note that compact Chern flat manifolds were classified by Boothby \cite{Boothby} in 1958 as the set of all compact quotients of complex Lie groups (equipped with left-invariant metrics compatible with the complex structure). The classification problem for compact Hermitian manifolds with flat Riemannian (Levi-Civita) connection, on the other hand, is still an open question when  $n\geq 4$. By Bieberbach's Theorem, this amounts to classifying on the flat torus $T^{2n}_{\mathbb R}$ all complex structures compatible  with the flat metric. For $n\leq 2$, only the standard complex structures occur (namely, the complex $n$-tori). For $n\geq 3$, lots of non-K\"ahlerian complex structures appear. The $n=3$ case was solved by Gabriel Khan, Bo Yang and the second named author \cite{KYZ} in 2017.

The above conjecture has also a stronger version, namely, assuming $H=f$ (or $H^c=f$) where $f$ is a function on $M^n$. The conjecture states that when $f$ is not identically zero, then $g$ must be K\"ahler, while when $f$ is identically zero, the conjecture states that $R$ (or $R^c$) must be zero.

In complex dimension $n=2$, the (stronger version of the) conjecture holds. In the $c\leq 0$ case, it was solved by Balas and Gauduchon \cite{BG} (see also \cite{Balas}) in 1985 for $\nabla^c$ and by Sato and Sekigawa \cite{SS} in 1990 for $\nabla$. The general case for $n=2$ (the stronger version, for both $\nabla^c$ and $\nabla$) was solved by Apostolov, Davidov, Muskarov \cite{ADM} in 1996, as a corollary of their beautiful classification theorem for compact self-dual Hermitian surfaces.

For $n\geq 3$, the conjecture is still largely open. The Chern case saw some partial development in recent years. In \cite{Tang} K. Tang  confirmed the conjecture under the additional assumption that $g$ is Chern K\"ahler-like (namely $R^c$ obeys all K\"ahler symmetries). In \cite{CCN}, H. Chen, L. Chen and X. Nie proved the conjecture under the additional assumption that $g$ is locally conformally K\"ahler and $c\leq 0$. In \cite{ZhouZ}, W. Zhou and the second named author proved that any compact balanced threefold with zero {\em real bisectional curvature}, a notion introduced in \cite{YangZ} which is slightly stronger than $H^c$, must be Chern flat. Also, in \cite{LZ} and \cite{RZ}, the second named author and collaborators confirmed the conjecture for $\nabla^c$ under the additional assumption that either $(M^n,g)$ is a complex nilmanifold with nilpotent complex structure $J$ (in the sense of \cite{CFGU}), or $g$ is Strominger K\"ahler-like (namely, $R^s$ obeys all K\"ahler symmetries).

For Strominger connection $\nabla^s$, which is the unique connection on a given Hermitian manifold $(M^n,g)$ satisfying $\nabla^sg=0$, $\nabla^sJ=0$,  and with totally skew-symmetric torsion, one could also ask the question of understanding its spaces forms. Mimicking the above folklore conjecture, we propose the following

\begin{conjecture} \label{conj2}
Let $(M^n,g)$ be a compact Hermitian manifold with $n\geq 2$. Assume that $H^s=c$ where $c$ is a constant. If $c\neq 0$, then $g$ is K\"ahler.
\end{conjecture}

Note that the second part of Conjecture \ref{conj1} does not hold for Strominger connection $\nabla^s$, namely, $H^s=0$ does not imply $R^s=0$ in general, as illustrated by isosceles Hopf manifold of dimension $n\geq 3$ which we will see in the next section. So instead one can ask the following question:

\vspace{0.25cm}

\noindent {\bf Question 1.} \label{question1}
\emph{ What kind of compact Hermitian manifold $(M^n,g)$ will have $H^s=0$ but  $R^s\not\equiv 0$?}

\vspace{0.25cm}

We want to understand {\em (weak) Strominger space forms,} which means compact Hermitian manifold $(M^n,g)$ whose $\nabla^s$ has (pointwise) constant holomorphic sectional curvature, namely, $H^s=c$ for a constant $c$ (or $H^s=f$ for a function $f\in C^{\infty }(M)$). When $g$ is K\"ahler, $f$ must be a constant by Schur's Lemma, and $(M^n,g)$ is a complex space form. So the goal is to understand the non-K\"ahler ones.

One class of examples of (non-K\"ahler)  Strominger space forms are the (compact) {\em Strominger flat manifolds,} namely compact Hermitian manifold with $R^s=0$. In this case we certainly have $H^s=0$. Strominger flat manifolds were classified by Qingsong Wang, Bo Yang, and the second named author in \cite{WYZ}, as quotients of {\em Samelson spaces} (namely, even-dimensional Lie groups with bi-invariant metrics and compatible left-invariant complex structures).

In the next section, we shall see that all {\em isosceles Hopf manifolds} satisfy $H^s=0$. They are Strominger flat when and only when $n=2$. So all isosceles Hopf manifolds of dimension $n\geq 3$ satisfy the description in Question \ref{question1}.

The main purpose of this paper is to understand all weak Strominger space forms in dimension $n=2$, and confirm Conjecture \ref{conj2} in the surface case:

\begin{theorem} \label{thm1}
Let $(M^2,g)$ be a compact Hermitian surface with $H^s=f$, where $f$ is a function on $M^2$. Then either $g$ is K\"ahler (in which case $f$ is a constant and the surface is a complex space form), or $M^2$ is an isosceles Hopf surface and $g$ is an admissible metric.
\end{theorem}

Recall that Hopf surfaces are compact complex surfaces whose universal cover is ${\mathbb C}^2\setminus \{ 0\}$. An {\em isosceles Hopf surface} means $M^2_{\phi}= {\mathbb C}^2\setminus \{ 0\} / \langle \phi \rangle $, where $\phi(z_1,z_2)=(a_1z_1,a_2z_2)$, with $0<|a_1|=|a_2|<1$. Here $(z_1,z_2)$ is the Euclidean coordinate, and let $g_0$ denotes the Euclidean metric whose K\"ahler form is given by
$$\omega_0 = \sqrt{-1} \big( dz_1\wedge d\overline{z}_1+ dz_2\wedge d\overline{z}_2 \big) = \sqrt{-1} \partial \overline{\partial} |z|^2$$
where $|z|^2= |z_1|^2+ |z_2|^2$. For any symmetric $2\times 2$ matrix $A$ satisfying
\begin{equation} \label{eq:admissible}
A\overline{A}<\frac{1}{4}I, \ \ \ \ \ \ DAD=|a_1|^2A,
\end{equation}
where $D=\mbox{diag}\{a_1, a_2\}$,  consider the function $\xi_A$ on ${\mathbb C}^2\setminus \{ 0\}$ defined by
\begin{equation} \label{eq:xi}
\xi_A = |z|^2 + \,^t\!zAz + \overline{^t\!zAz}.
\end{equation}
Condition (\ref{eq:admissible}) guarantees that $\xi_A>0$ and $\phi^{\ast} \xi_A = |a_1|^2\xi_A$. So the Hermitian metric $g_A=\frac{1}{\xi_A}g_0$ is invariant under $\phi$ hence descends down to $M^2_{\phi}$. A Hermitian metric $g$ on $M^2_{\phi}$ is called an {\em admissible metric,} if it is a constant multiple of $g_A$ for some $A$ satisfying (\ref{eq:admissible}). We will show in \S 2 that the Strominger holomorphic sectional curvature of $g_A$ is
\begin{equation} \label{eq:Hsofxi}
H^s = - \frac{1}{\xi_A} \big( 4\,^t\!zA\,\overline{Az} + \,^t\!zAz + \overline{^t\!zAz}\big),
\end{equation}
so any admissible metric on an isosceles Hopf surface has pointwise constant Strominger holomorphic sectional curvature, hence is a weak space form. The same is true in higher dimensions. When $A=0$, we get the standard Hopf metric $g_h=\frac{1}{|z|^2}g_0$, which has $H^s\equiv 0$. Any admissible metric is globally conformal to $g_h$ on $M^2_{\phi}$ as $\frac{\xi_A}{|z|^2}$ descends down to $M^2_{\phi}$.

Note that for any $A\neq 0$, the $H^s$ of $g_A$ can never be a constant function, therefore we obtain  the following:

\begin{corollary} \label{cor2}
Let $(M^2,g)$ be a compact Hermitian surface with $H^s=c$ where $c$ is a constant. If $c\neq 0$, then $g$ is K\"ahler. If $c=0$, the either $g$ is K\"ahler, or $(M^2,g)$ is an isosceles Hopf surface with the (scaling of) standard metric $g_h$, which has $R^s=0$.
\end{corollary}

This confirms Conjecture \ref{conj2} for $n=2$, and also answers Question \ref{question1} in the negative way for $n=2$: in dimension $2$ the only non-K\"ahler compact Hermitian surfaces with $H^s=0$ are isosceles Hopf surfaces, which are Strominger flat.

At this point, we know little about the  higher dimensional situations, and here we will satisfy ourselves with the following two special situations.

\begin{theorem} \label{thm3}
Let $(M^n,g)$ be a complex nilmanifold with nilpotent $J$, namely, the universal covering space is a nilpotent Lie group $G$ equipped with a left-invariant complex structure $J$ and a compatible left-invariant metric. We assume that $J$ is nilpotent in the sense of \cite{CFGU}. If $H^s=f$ where $f$ is a function on $M^n$, then $f\equiv 0$, $g$ is K\"ahler flat,  and $G$ is abelian.
\end{theorem}

In other words there is no non-K\"ahler Stromionger weak space form amongst complex nilmanifolds with nilpotent $J$.

\begin{theorem} \label{thm4}
Let $(M^n,g)$ be a compact Strominger K\"ahler-like (SKL) manifold with $H^s=f$ where $f$ is a function on $M^n$. Then either $g$ is K\"ahler (hence a complex space form), or $f\equiv 0$ and $g$ is Strominger flat.
\end{theorem}

That is, the only non-K\"ahler Strominger space forms amongst SKL manifolds are the Strominger flat manifolds. SKL manifolds were studied in  \cite{YZZ}, \cite{ZZ}, and \cite{ZZ1}, where classification theorem in $n=2$ and $n=3$  as well as some structural results in general dimensions were obtained. SKL manifolds form an interesting subset of pluriclosed manifolds.

The paper is organized as follows. In \S 2, we will set up the notations and collect some known results from literature, which will prepare us for the proofs of the theorems stated in this introduction. We will also give the computations for isosceles Hopf manifolds. In \S 3 we will proof the main result, Theorem \ref{thm1}, and in \S 4, we will discuss partial results in higher dimensions and prove Theorem \ref{thm3} and \ref{thm4}.

\vspace{0.3cm}

\section{Preliminaries}

Let $(M^n,g)$ be a Hermitian manifold and denote by  $\omega$ the  K\"ahler form associated with $g$. Denote by $\nabla$,  $\nabla^c$, $\nabla^s$ the Levi-Civita (Riemannian), Chern, Strominger (also known as Bismut) connection, respectively.

Fix any $p \in M^n$, let $\{ e_1, \ldots , e_n\} $ be a frame of $(1,0)$-tangent vectors of $M^n$ in a neighborhood of $p$, with $\{ \varphi_1, \ldots , \varphi_n\}$ being its dual coframe of $(1,0)$-forms. We will also write  $e=\,^t\!(e_1, \ldots , e_n)$ and $\varphi = \,^t\!( \varphi_1, \ldots , \varphi_n)$ and view them as column vectors. Let  $\langle \ , \rangle $ be the (real) inner product given by the Hermitian metric $g$, and extend it bi-linearly over ${\mathbb C}$. Following the notations of  \cite{YZ,YZ1}, let us write
\[\begin{cases}\nabla e = \theta_1 e + \overline{\theta_2 }\,\overline{e} \\
\nabla \overline{e} = \theta_2 e + \overline{\theta_1}\,\overline{e}
\end{cases}\]
where the matrices of connection and curvature of $\nabla $ under the frame $\,^t\!(e,\overline{e})$ are given by
$$ \hat{\theta } = \left[ \begin{array}{ll} \theta_1 & \overline{\theta_2 } \\ \theta_2 & \overline{\theta_1 }  \end{array} \right] , \ \  \  \hat{\Theta } = \left[ \begin{array}{ll} \Theta_1 & \overline{\Theta}_2  \\ \Theta_2 & \overline{\Theta}_1   \end{array} \right], $$
 where
\begin{equation}
\label{structure-r}
\left\{ \begin{array}{lll} \Theta_1 = d\theta_1 -\theta_1 \wedge \theta_1 -\overline{\theta_2} \wedge \theta_2,  \\
\Theta_2 = d\theta_2 - \theta_2 \wedge \theta_1 - \overline{\theta_1 } \wedge \theta_2,   \\
d\varphi = - \ ^t\! \theta_1 \wedge \varphi - \ ^t\! \theta_2
\wedge \overline{\varphi } \end{array} \right.
\end{equation}
Similarly, let $\theta$ and $\tau$ be respectively the connection matrix and torsion column vector under $e$ for the Chern connection $\nabla^c$, then the structure equations and Bianchi identities are
\begin{equation}
\label{structure-c} \left\{ \begin{array}{llll} d \varphi = - \ ^t\!\theta \wedge \varphi + \tau,   \\
d  \theta = \theta \wedge \theta + \Theta.  \\
d \tau = - \ ^t\!\theta \wedge \tau + \ ^t\!\Theta \wedge \varphi, \\
d  \Theta = \theta \wedge \Theta - \Theta \wedge \theta. \end{array} \right.
\end{equation}
Finally, let $\theta^s$, $\Theta^s$ be the matrices of connection and curvature under $e$ of the Strominger connection $\nabla^s$, then we have
\begin{equation}
\label{structure-s}
\Theta^s = d\theta^s -\theta^s \wedge \theta^s.
\end{equation}

As is well-known, the entries of the curvature matrix $\Theta$ are all $(1,1)$-forms, while the entries of the column vector $\tau $ are all $(2,0)$-forms, under any frame $e$. Consider the $(2,1)$ tensor $\gamma =\nabla^1-\nabla^c$. Here $\nabla^1$ is the Hermitian projection of $\nabla$, sending $e$ to $\theta_1 e$. It is called {\em Lichnerowicz connection} of $g$ in some literature. For convenience, we will also use $\gamma$ to denote the matrix of $1$-forms representing $\gamma$ under $e$, namely, we also write $\gamma = \theta_1-\theta$. Denote by $\gamma = \gamma' + \gamma ''$ the decomposition into the $(1,0)$ and $(0,1)$-parts. Note that $\overline{\theta_2}$ is the matrix under $e$ of the tensor $\nabla - \nabla^1$.  It follows from \cite{YZ} that when $e$ is unitary, the matrices $\gamma$ and $\theta_2$ are given by
\begin{equation}\label{eq:4}
\gamma_{ij} = \sum_k \{ T^j_{ik}\varphi_k  - \overline{T^i_{jk}} \,\overline{\varphi}_k \} , \ \ \ (\theta_2)_{ij} = \sum_k \overline{T^k_{ij} } \,\varphi_k ,
\end{equation}
where $T^k_{ij}$, satisfying $T^k_{ji}=-T^k_{ij}$, are the components of the Chern torsion, given by
\[\tau_k=\sum_{i,j=1}^n T^k_{ij} \varphi_i \varphi_j = 2 \sum_{i<j} T^k_{ij} \varphi_i \varphi_j .\]
It is also well-known (cf. \cite{YZ}) that
\begin{equation} \label{eq:5}
\theta^s=\theta + 2\gamma .
\end{equation}
The {\em Gauduchon torsion $1$-form}, which is a global $(1,0)$-form on $M^n$, is defined by $\eta = \mbox{tr}(\gamma')$, or equivalently, by $\partial \omega^{n-1} = -2 \eta \wedge \omega^{n-1}$ where $\omega = \sqrt{-1}\,^t\!\varphi \wedge \overline{\varphi}$ is the K\"ahler form of $g$. The components of $\eta$ under $e$ are
$$ \eta = \sum_k \eta_k \varphi_k, \ \ \ \ \ \eta_k = \sum_{i} T^i_{ik}. $$
By the same proof of \cite[Lemma 4]{YZ}, we have the following

\begin{lemma}\label{lemma1}
Let $(M^n,g)$ be a Hermitian manifold. Let $D$ be any Hermitian connection on $M$, namely, one satisfying $Dg=0$ and $DJ=0$. Then for any given $p\in M$, there exists a unitary frame $e$ of $(1,0)$-tangent vectors in a neighborhood of $p$, such that the connection matrix $\theta^D$ for $D$ vanishes at $p$.
\end{lemma}

We often will use the above lemma for $D=\nabla^c$ or $\nabla^s$ in our later proofs, so at $p$ first order covariant derivatives with respect to $\nabla^c$ (or $\nabla^s$) simply become first order partial (directional) derivatives.

Next let us fix the curvature notation. We have
$$ \Theta_{ij} = \sum_{k,\ell =1}^n R^c_{k\overline{\ell} i \overline{j}} \varphi_k \wedge \overline{\varphi}_{\ell}, \ \ \ \ \Theta^s_{ij} = \sum_{k,\ell =1}^n \big( R^s_{k\ell i \overline{j}} \,\varphi_k \wedge \varphi_{\ell}  + R^s_{\overline{k}\overline{\ell} i \overline{j}} \,\overline{\varphi}_k \wedge \overline{\varphi}_{\ell} + R^s_{k\overline{\ell} i \overline{j}} \,\varphi_k \wedge \overline{\varphi}_{\ell} \big) , $$
and
\begin{eqnarray*}
 && (\Theta_1)_{ij} = \sum_{k,\ell =1}^n \big( R_{k\ell i \overline{j}} \,\varphi_k \wedge \varphi_{\ell}  + R_{\overline{k}\overline{\ell} i \overline{j}} \,\overline{\varphi}_k \wedge \overline{\varphi}_{\ell} + R_{k\overline{\ell} i \overline{j}} \,\varphi_k \wedge \overline{\varphi}_{\ell} \big) , \\
 && (\Theta_2)_{ij} = \sum_{k,\ell =1}^n \big( R_{k\ell \overline{i} \overline{j}} \,\varphi_k \wedge \varphi_{\ell}  + R_{k\overline{\ell}\overline{i} \overline{j}} \,\varphi_k \wedge \overline{\varphi}_{\ell} \big).
 \end{eqnarray*}
Note that by Gray's theorem, one always has $R_{ijk\ell}=0$. Given any $p\in M$, let $e$ be local unitary frame near $p$ so that $\theta|_p=0$. Then by the structure equations, at $p$ we have $\overline{\partial}\varphi|_p=0$ and
\begin{eqnarray*}
(\Theta_1)^{1,1} - \Theta & = & (d\gamma - \gamma \gamma - \overline{\theta}_2 \theta_2)^{1,1} \ = \ \overline{\partial} \gamma' -  \partial {\gamma'}^{\ast} + \gamma' \gamma'^{\ast} +  \gamma'^{\ast}\gamma' - \overline{\theta}_2 \theta_2, \\
(\Theta^s)^{1,1} - \Theta & = & (2d\gamma - 4\gamma \gamma )^{1,1} \ = \ 2\overline{\partial} \gamma' - 2 \partial {\gamma'}^{\ast} + 4\gamma' \gamma'^{\ast} + 4 \gamma'^{\ast}\gamma'.
\end{eqnarray*}
From this we get the relationship between $R_{k\overline{\ell}i\overline{j}}=(\Theta_1)_{ij}(e_k, \overline{e}_{\ell})$ or  $R^s_{k\overline{\ell}i\overline{j}}=\Theta^s_{ij}(e_k, \overline{e}_{\ell})$ with $R^c_{k\overline{\ell}i\overline{j}}=\Theta_{ij}(e_k, \overline{e}_{\ell})$:

\begin{lemma}\label{lemma2}
Let $(M^n,g)$ be a Hermitian manifold. Under any local unitary frame $e$, it holds
\begin{eqnarray}
R_{k\overline{\ell}i\overline{j}} - R^c_{k\overline{\ell}i\overline{j}} &  = & -T_{ik,\overline{\ell}}^j - \overline{T^i_{j\ell , \overline{k}}}   + \sum_r \big(  T^{r}_{ik}\overline{T^r_{j\ell}} - T^{j}_{kr}\overline{T^i_{\ell r}} - T^{\ell}_{ir}\overline{T^k_{j r}}  \big) \label{RRc} \\
R^s_{k\overline{\ell}i\overline{j}} - R^c_{k\overline{\ell}i\overline{j}} &  = &   -2T_{ik,\overline{\ell}}^j - 2 \overline{T^i_{j\ell , \overline{k}}}   + 4\sum_r \big(   T^{r}_{ik}\overline{T^r_{j\ell}} - T^{j}_{kr}\overline{T^i_{\ell r}}  \big)  \label{RsRc}
\end{eqnarray}
for any $i$, $j$, $k$, $\ell$, where the indices after comma mean covariant derivatives with respect to $\nabla^c$.
\end{lemma}

Next let us recall the formula of Weyl curvature tensor for a Hermitian surface $(M^2,g)$. One has $W=W^+ + W^-$, and $g$ is said to be {\em self-dual} if $W^-=0$. Let $\{ e_1, e_2\}$ be a local unitary frame of $(M^2,g)$. Then
$$ \{ e_1\wedge \overline{e}_2, e_2\wedge \overline{e}_1, \frac{1}{\sqrt{2}}(e_1\wedge \overline{e}_1 - e_2\wedge \overline{e}_2) \} $$
form a basis of the complexification of $\Lambda^2_-$, the space on which the Hodge star operator is minus identity. Let us call this basis the standard basis associated with the unitary frame $e$. In \cite{ADM}, it was proved that

\begin{lemma}[\cite{ADM}]  \label{lemma3}
Let $(M^2,g)$ be a Hermitian surface, and $e$ a local unitary frame. Then under the standard basis associated with $e$, the components of $W^-$, the anti-self dual part of the Weyl tensor, are given by
\begin{equation} \label{Wminus}
\left\{ \begin{array}{lll} W^-_1=R_{1\overline{2}1\overline{2}}, \\ W^-_2 = \frac{1}{\sqrt{2}} \big( R_{1\overline{2}2\overline{2}} - R_{1\overline{2}1\overline{1}} \big), \\ W^-_3 = \frac{1}{6} \big( 2R_{1\overline{2}2\overline{1}}    + 2 R_{1\overline{1}2\overline{2}} - R_{1\overline{1}1\overline{1}} - R_{2\overline{2}2\overline{2}} \big) , \end{array} \right.
\end{equation}
\end{lemma}

The last formula in the lemma is a combination of the formula $W^-_3 =  R_{1\overline{2}2\overline{1}} - \frac{\sigma }{12}$
given in \cite{ADM} and the formula
$$ \sigma = 2\big( R_{1\overline{1}1\overline{1}} + R_{2\overline{2}2\overline{2}} - 2 R_{1\overline{1}2\overline{2}} + 4R_{1\overline{2}2\overline{1}} \big) $$
under the unitary frame $e$. Here $\sigma$ is the scalar curvature of $g$ (denoted as $\tau$ in \cite{ADM}).

\vspace{0.2cm}

In the remaining part of this section, let us compute the Strominger curvature for isosceles Hopf manifolds. We begin with a simple computation, which will also be used in the next section. Let $\Omega \subseteq {\mathbb C}^n$ be a domain, namely, a connected open subset. Denote by $z=(z_1, \ldots , z_n)$ the Euclidean coordinate and by $g_0$ the Euclidean metric. Suppose $\xi \in C^{\infty}(\Omega)$ is a smooth positive function, and consider the Hermitian metric $g=\frac{1}{\xi}g_{0}$ on $\Omega$.

Write $h = \sqrt{\xi}$, $e_i=h\frac{\partial}{\partial z_i}$, $\varphi_i=\frac{1}{h}dz_i$. Then $e$ becomes a unitary frame of $(\Omega , g)$, with $\varphi$ the dual coframe. Under this frame, we have
$$ \theta_{ij}= \frac{1}{h} (\overline{\partial} - \partial )h \,\delta_{ij} , \ \ \ \ \ T^j_{ik} = \delta_{ij}h_k - \delta_{kj}h_i, $$
where $h_i=\frac{\partial h}{\partial z_i}$. From this, one gets the expression of $\gamma$ and
$$ \theta^s_{ij} = \frac{1}{h}(\partial h - \overline{\partial}h) \,\delta_{ij}+ \frac{2}{h} \big( h_{\overline{j}} d\overline{z}_i - h_i dz_j \big) . $$
Taking the $(1,1)$-part of $\Theta^s = d\theta^s - \theta^s \wedge \theta^s$, a straight forward computation leads to the following
\begin{eqnarray}
R^s_{k\overline{\ell} i \overline{j}} & = & - 4 \sum_r |h_r|^2\delta_{i\ell} \delta_{kj} -2\delta_{ij} \big( h h_{k\overline{\ell}} - h_k h_{\overline{\ell}} \big) + 4 \delta_{k\ell} h_i h_{\overline{j}} + \nonumber \\
&& + \ 2\delta_{i\ell} \big( h h_{k\overline{j}} - h_k h_{\overline{j}} \big) + 2\delta_{kj} \big( hh_{i\overline{\ell}} - h_i h_{\overline{\ell}} \big) \label{Rs}
\end{eqnarray}
for any $1\leq i,j,k,\ell \leq n$, where $h_{i\overline{j}}$ stands for $\frac{\partial^2h}{\partial z_i \partial \overline{z}_j}$.

For tensor $P_{i\overline{j}k\overline{\ell}}$, let us denote by $\widehat{P}_{i\overline{j}k\overline{\ell}} $ its {\em symmetrization,} namely,
\begin{equation}
 \widehat{P}_{i\overline{j}k\overline{\ell}}  : = \frac{1}{4} \big( P_{i\overline{j}k\overline{\ell}} + P_{k\overline{j}i\overline{\ell}} + P_{i\overline{\ell}k\overline{j}} + P_{k\overline{\ell}i\overline{j}}\big) .
 \end{equation}
Taking the symmetrization of  (\ref{Rs}) and using the fact that $hh_{i\overline{j}}+ h_ih_{\overline{j}}=\frac{1}{2}\xi_{i\overline{j}}$,  we get the expression of the symmetrization of Strominger curvature tensor of the metric $g=\frac{1}{\xi}g_0$ under the unitary frame:
\begin{equation} \label{Rshat}
\widehat{R}^s_{i\overline{j}k\overline{\ell}} = -\frac{1}{2\xi}\sum_r|\xi_r|^2 (\delta_{i\ell} \delta_{kj} + \delta_{ij}  \delta_{k\ell} ) + \frac{1}{4} \{ \delta_{ij} \xi_{k\overline{\ell}} + \delta_{k\ell }\xi_{i\overline{j}} +\delta_{i\ell }\xi_{k\overline{j}} + \delta_{kj }\xi_{i\overline{\ell}}    \} .
\end{equation}

 Note that the Strominger holomorphic sectional curvature is pointwise constant, namely $H^s=f$ for some real-valued function $f$, if and only if  $\widehat{R}^s_{i\overline{j}k\overline{\ell}} = \frac{c}{2} \big( \delta_{ij} \delta_{k\ell} + \delta_{i\ell} \delta_{kj} \big)$ under any unitary frame. So we have the following observation:

\begin{lemma} \label{lemma4}
Let $\Omega \subseteq {\mathbb C}^n$ be an open connected subset, $\xi$ a smooth positive function on $\Omega$, and $g=\frac{1}{\xi}g_{0}$ a Hermitian metric conformal to the Euclidean metric $g_{0}$. Then $(\Omega , g)$ has pointwise constant Strominger holomorphic sectional curvature $H^s=f$ if and only if $\xi$ satisfies
\begin{equation}
\xi_{i\overline{j}} = \lambda \delta_{ij}, \ \ \ \lambda = f+ \frac{1}{\xi} \sum_r|\xi_r|^2 = \mbox{constant}
\end{equation}
for any $1\leq i,j\leq n$. Here $\xi_r = \frac{\partial \xi}{\partial z_r}$ and $\xi_{i\overline{j}}=\frac{\partial^2 \xi }{\partial z_i \partial \overline{z}_j}$.
\end{lemma}

Now let us apply the above observation to the special case of {\em isosceles Hopf manifolds.} Let $\Omega = {\mathbb C}^n\setminus \{0\}$, $n\geq 2$. Write $|z|^2=|z_1|^2+\cdots +|z_n|^2$. For $\xi= |z|^2$, we compute that
$$ \xi_i = \overline{z}_i, \ \ \ \xi_{i\overline{j}}=\delta_{ij}. $$
So $\xi$ satisfies the equation in Lemma \ref{lemma4} with $\lambda =1$, which leads to $f\equiv 0$. So we know that the Hermitian metric $g_h$ with K\"ahler form
$ \omega_{g_h} = \sqrt{-1} \, \frac{\partial \overline{\partial} |z|^2} {|z|^2}$
has Strominger holomorphic sectional curvature $H^s\equiv 0$.

Recall that an {\em isosceles Hopf manifold} is $M^n_{\phi}$   is the quotient of ${\mathbb C}^n\setminus \{0\}$ by the infinite cyclic group generated by
$$\phi : (z_1, z_2, \cdots , z_n) \mapsto (a_1z_1, a_2z_2, \ldots , a_nz_n) $$
where $0<|a_1|=\cdots =|a_n|<1$. The metric $g_h$ (called the {\em standard Hopf metric}) with K\"ahler form  $\sqrt{-1} \, \frac{\partial \overline{\partial} |z|^2} {|z|^2}$ descends down to $M^n_{\phi}$ and by Lemma \ref{lemma4} we have

\begin{lemma} \label{lemma5}
For each $n\geq 2$, any isosceles Hopf manifold $(M^n,g_h)$ has vanishing Strominger holomorphic sectional curvature: $H^s=0$. It is Strominger flat (namely  $R^s=0$) if and only if $n=2$.
\end{lemma}

\begin{proof} We just need to check the last sentence.   When $n\geq 3$, apply the formula (\ref{Rs}) for $h=|z|$, we get $$R^s_{1\overline{1}2\overline{3}}=-2(hh_{2\overline{3}} -h_2h_{\overline{3}}) = 4h_2h_{\overline{3}}= \frac{\overline{z}_2z_3}{|z|^2} \neq 0.$$
So the Hopf manifold is not Strominger flat when $n\geq 3$. For $n=2$, by a straight forward computation we conclude that $R^s=0$, so the standard Hopf metric on any isosceles Hopf surface is Strominger flat.
\end{proof}

The above example $(M^n_{\phi},g_h)$ with $n\geq 3$ illustrates that there are Strominger space forms (with $H^s=0$) that are not Strominger flat.

Next, let $M=M^n_{\phi}$ and $\pi: {\mathbb C}^n \setminus \{ 0\} \rightarrow M$ be the projection map. Suppose $\tilde{g}$ is a Hermitian metric on $M$ conformal to $g_h$ so that $\tilde{g}$ has pointwise constant Strominger holomorphic sectional curvature. Write $\tilde{g} =\frac{1}{F}g_h$ for some positive function $F\in C^{\infty }(M)$, and let $\xi =\pi^{\ast}F \cdot |z|^2$. Then the metric  $\pi^{\ast}\tilde{g}=\frac{1}{\xi}g_{0}$ has $\tilde{H}^s=f$ for some real-valued smooth function $f$. Both $F$ and $f$ are invariant under $\phi$. By Lemma \ref{lemma4}, $\xi$ satisfies the condition $\xi_{i\overline{j}}=\lambda \delta_{ij}$ where $\lambda = f+ \frac{1}{\xi}|\xi_r|^2$ is a constant. This implies that $\xi-\lambda |z|^2$ is pluriclosed, hence is (twice of) the real part of a holomorphic function $\psi$ on ${\mathbb C}^n \setminus \{ 0\}$:
$$ \xi = \lambda |z|^2 + \psi + \overline{\psi}. $$
Let us write $a_i=a\rho_i$ for each $i$, where $a\in (0,1)$ and $|\rho_i|=1$. Then $\phi^{\ast} |z|^2 = a^2 |z|^2$,  $\phi^{\ast} \xi = a^2 \xi$, hence $\phi^{\ast}(\psi + \overline{\psi}) = a^2(\psi + \overline{\psi})$. Thus $\phi^{\ast}\psi - a^2\psi$ is a (pure imaginary) constant, so by adding a suitable constant to $\psi$ if necessary, we may assume that $\phi^{\ast}\psi = a^2\psi$. Expanding $\psi$ into Laurent series, we see that the condition implies that $\psi$ is a homogenous quadratic polynomial in $z$, namely, $\psi =  \,^t\!z A z$, where $A$ is a symmetric $n\times n$ matrix and we view $z$ as a column vector.

Since $\xi = \lambda |z|^2+ \,^t\!z A z +  \overline{\,^t\!z A z}$ is positive, we must have $\lambda >0$. For simplicity, let us replace $\tilde{g}$ by $\frac{1}{\lambda}\tilde{g}$, then $F$ and $\xi$ are scaled by the factor $\frac{1}{\lambda}$, so we may assume from now on that $\lambda =1$. We now have
\begin{equation} \label{eq:xi}
\xi = |z|^2+ \,^t\!z A z +  \overline{\,^t\!z A z}
\end{equation}
where $A$ is a symmetric $n\times n$ matrix. By Lemma \ref{lemma4}, we also have $1=f+ \frac{1}{\xi}\sum_r |\xi_r|^2$, which is equivalent to
\begin{equation} \label{eq:f}
 f= - \frac{1}{\xi}\big( 4\,\,^t\!z A\,\overline{Az} + \,^t\!z A z +  \overline{\,^t\!z A z} \big) .
\end{equation}
The condition $\phi^{\ast}\psi = a^2\psi$ means that $A$ must satisfy
\begin{equation} \label{eq:A1}
A_{ij}\rho_i\rho_j = A_{ij}, \ \ \ \forall \ i,j, \ \ \ \ \mbox{or equivalently,} \ \ \ DAD=A
\end{equation}
where $D=\mbox{diag}\{ \rho_1, \ldots , \rho_n\}$. This will guarantee that $DA\,\overline{AD}=A\overline{A}$, hence $\phi^{\ast}f=f$. For any given symmetric matrix $A$, by Chern's lemma there always exists a unitary matrix $U$ such that $\,^t\!UAU=\mbox{diag}\{ b_1, \ldots , b_n\}$ where $b_1\geq b_2 \geq \cdots \geq b_n\geq 0$. Note that $\{ b_1^2, \ldots , b_n^2\}$ are exactly the eigenvalues of the Hermitian matrix $A\overline{A}$.

For a constant $b\geq 0$, since $ |t|^2+ b(t^2+\overline{t^2})>0$ for any $t\in {\mathbb C}\setminus \{0\}$ if and only if $b< \frac{1}{2}$, we know that the condition $\xi>0$ also means that
\begin{equation} \label{eq:A2}
A \overline{A} < \frac{1}{4} I  .
\end{equation}
In summary, we have the following:

\begin{proposition} \label{prop1}
Let $M^n_{\phi}$ be the quotient of ${\mathbb C}^n\setminus \{ 0 \}$ by the infinite cyclic group generated by
$$ \phi (z_1, \ldots , z_n) = a(\rho_1z_1, \ldots , \rho_nz_n), $$
where $0<a<1$ and each $|\rho_i|=1$. Then for any  symmetric matrix $A$ satisfying (\ref{eq:A1}) and (\ref{eq:A2}), the metric $g_A:=\frac{1}{\xi}g_{0}$ where $\xi$ is given by (\ref{eq:xi}) descends down to the isosceles Hopf manifold $M^n_{\phi}$, and $g_A$ has pointwise constant Strominger holomorphic sectional curvature $f$ given by (\ref{eq:f}).

Conversely, any Hermitian metric on $M^n_{\phi}$ conformal to the standard metric $g_h$ and having pointwise constant Strominger holomorphic sectional curvature must be a constant multiple of one of those $g_A$.
\end{proposition}

We will call (any constant multiple of) these $g_A$ {\em admissible metrics} on the isosceles Hopf manifold $M^n_{\phi}$. When $A=0$, the corresponding metric is simply the standard Hopf metric $g_h$, which has $H^s\equiv 0$.

If $g_A$ has constant $H^s$, say $H^s=c$, then by (\ref{eq:f}) we get
$$ \big( 4\,\,^t\!z A\,\overline{Az} + \,^t\!z A z +  \overline{\,^t\!z A z} \big) = (-c) \big( |z|^2 + \,^t\!z A z +  \overline{\,^t\!z A z} \big) $$
for any $0\neq z\in {\mathbb C}^n$. This forces $-c=1$, so $4\,\,^t\!z A\,\overline{Az}=|z|^2$ which leads to $A\,\overline{A}=\frac{1}{4}I$, contradicting with condition (\ref{eq:A2}). So we know that:

{\em Any admissible metric $g_A$ with $A\neq 0$ cannot have constant $H^s$.}

Therefore we get a proof of Corollary \ref{cor2} under the assumption that we have already established Theorem \ref{thm1}.

\vspace{0.3cm}

\section{The surface case}

We will prove Theorem \ref{thm1} in this section. As in \cite{ADM}, the first step is to show that any Hermitian surface with pointwise constant Strominger holomorphic sectional curvature must be self-dual:

\begin{proposition} \label{prop2}
Let $(M^2,g)$ be a Hermitian surface. Assume that the Strominger holomorphic sectional curvature is pointwise constant, namely $H^s=f$ where $f$ is a real-valued function on $M^2$. Then $W^-=0$, namely, $M$ is self-dual.
\end{proposition}

\begin{proof}
Again denote by $\widehat{P}_{i\overline{j}k\overline{\ell}} $ the symmetrization of $P_{i\overline{j}k\overline{\ell}}$. The assumption $H^s=f$ is equivalent to $\widehat{R}^s_{i\overline{j}k\overline{\ell}} = \frac{f}{2}\big( \delta_{ij}\delta_{k\ell} + \delta_{i\ell} \delta_{kj} \big)$ under any unitary frame $e$. In particular,
\begin{equation*}
\widehat{R}^s_{1\overline{2}1\overline{2}} = \widehat{R}^s_{1\overline{1}1\overline{2}} = \widehat{R}^s_{1\overline{2}2\overline{2}} =0, \ \ \ \ \widehat{R}^s_{1\overline{1}1\overline{1}} = \widehat{R}^s_{2\overline{2}2\overline{2}} =f.
\end{equation*}
Subtracting the two equalities of Lemma \ref{lemma2} yields
\begin{equation*}\label{RRs}
R_{k\overline{\ell}i\overline{j}} - R^s_{k\overline{\ell}i\overline{j}} = T_{ik,\overline{\ell}}^j + \overline{T^i_{j\ell , \overline{k}}}   + \sum_r \big(  -3 T^{r}_{ik}\overline{T^r_{j\ell}} +3 T^{j}_{kr}\overline{T^i_{\ell r}} - T^{\ell}_{ir}\overline{T^k_{j r}}  \big).
\end{equation*}
Taking symmetrization, we get
\begin{equation}\label{hatRRs}
\widehat{R}_{k\overline{\ell}i\overline{j}} - \widehat{R}^s_{k\overline{\ell}i\overline{j}} = 2 \sum_r \big( T^{j}_{ir}\overline{T^k_{\ell r}} + T^{\ell }_{kr}\overline{T^i_{jr}} + T^{j}_{kr}\overline{T^i_{\ell r}} + T^{\ell}_{ir}\overline{T^k_{j r}}  \big).
\end{equation}
Since $n=2$, it gives us  $\widehat{R}_{1\overline{2}1\overline{2}} -\widehat{R}^s_{1\overline{2}1\overline{2}} = 8\sum_{r=1}^2 T^{2 }_{1r}\overline{T^1_{2r}} = 0$. Hence
\begin{equation}\label{eq:W-1}
 R_{1\overline{2}1\overline{2}} = \widehat{R}_{1\overline{2}1\overline{2}} = \widehat{R}^s_{1\overline{2}1\overline{2}} = 0.
 \end{equation}
Since $R_{xyzw}=R_{zwxy}$, we have
$$ R_{1\overline{1}1\overline{2}} = \frac{1}{2}\big( R_{1\overline{1}1\overline{2}} + R_{1\overline{2}1\overline{1}} \big) = \widehat{R}_{1\overline{1}1\overline{2}} = \widehat{R}^s_{1\overline{1}1\overline{2}} + 4 \sum_{r=1}^2  T^{2}_{1r}\overline{T^1_{1 r}} = 4 T^{2}_{12}\overline{T^1_{12}}.$$
Similarly, we have $ R_{1\overline{2}2\overline{2}} = 4 \sum_{r=1}^2 T^{2}_{2r}\overline{T^1_{2r}} = 4 T^{2}_{12}\overline{T^1_{12}}$. Therefore
\begin{equation}\label{eq:W-2}
 R_{1\overline{1}1\overline{2}} - R_{1\overline{2}2\overline{2}} = 0.
 \end{equation}
By (\ref{hatRRs}), we have
$$ R_{1\overline{1}1\overline{1}} = \widehat{R}^s_{1\overline{1}1\overline{1}} + 8\sum_r |T^1_{1r}|^2 = f+ 8 |T^1_{12}|^2.$$
Similarly, $R_{2\overline{2}2\overline{2}} = f + 8  |T^2_{12}|^2$. Also by (\ref{hatRRs}), we have
$$ \widehat{R}_{1\overline{1}2\overline{2}} = \widehat{R}^s_{1\overline{1}2\overline{2}} + 2\sum_r \big( |T^2_{1r}|^2 + |T^1_{2r}|^2\big) = \frac{f}{2} + 2\big( |T^1_{12}|^2 + |T^2_{12}|^2\big). $$
From these, we get
\begin{equation}\label{eq:W-3}
 4\widehat{R}_{1\overline{1}2\overline{2}} - R_{1\overline{1}1\overline{1}} - R_{2\overline{2}2\overline{2}} = 0.
 \end{equation}
Since $ 4\widehat{R}_{1\overline{1}2\overline{2}} = 2\big( R_{1\overline{1}2\overline{2}} + R_{1\overline{2}2\overline{1}} \big) $, by Lemma \ref{lemma3} and (\ref{eq:W-1})-(\ref{eq:W-3}), we know that $W^-=0$. That is, the Hermitian surface $(M^2, g)$ is self-dual. This completes the proof of Proposition \ref{prop2}.
\end{proof}

Next, we consider the behavior of Strominger holomorphic sectional curvature under conformal changes. Let $(M^n,g)$ be a Hermitian manifold, and  $\tilde{g}=e^{2u}g$ a Hermitian metric on $M^n$ conformal to $g$. Here $u\in C^{\infty}(M)$ is a smooth real-valued function. Let $e$ be a local unitary frame of $g$, with dual coframe $\varphi$. Then $\tilde{e}=e^{-u}e$ and $\tilde{\varphi}=e^u\varphi$ become unitary frame and dual coframe for $\tilde{g}$. Under respective unitary frames, the connection matrix and torsion for Chern connection are related by (cf. in the proof of Theorem 3 in \cite{YZZ}):
\begin{equation}
\tilde{\theta} = \theta +(\partial u - \overline{\partial} u)I, \ \ \ \tilde{T}^j_{ik} = e^{-u}\{T^j_{ik} + u_i \delta_{jk} - u_k \delta_{ji} \}
\end{equation}
where $u_k=e_k(u)$. From this, we get the expression for $\gamma$, $\tilde{\gamma}$ hence
\begin{equation}
\tilde{\theta}^s_{ij} - \theta^s_{ij} = 2 u_i \varphi_j - \partial u \delta_{ij} - 2 u_{\overline{j}} \overline{\varphi}_i + \overline{\partial} u\delta_{ij} .
\end{equation}
Fix a point $p\in M$. Let $e$ be local unitary frame of $(M^n,g)$ near $p$ such that $\theta^s|_p=0$. Then by structure equation $\overline{\partial}\varphi = -2\overline{\gamma'} \varphi$ at $p$, so we have
\begin{equation} \label{Thetas11}
(\tilde{\Theta}^s_{ij})^{1,1} - ( \Theta^s_{ij})^{1,1} =  2\overline{\partial}(u_i\varphi_j) - 2 \partial (u_{\overline{j}} \overline{\varphi}_i) + 2 \partial \overline{\partial} u \delta_{ij} + 4u_i u_{\overline{j}} \varphi_r \overline{\varphi}_r - 4|u_r|^2 \varphi_j \overline{\varphi}_i.
\end{equation}
From this, a straight forward computation leads to the following

\begin{proposition} \label{prop3}
Let $(M^n,g)$ be a Hermitian manifold, $\tilde{g}=e^{2u}g$ a metric conformal to $g$. Under any unitary frames $e$ for $g$ and $\tilde{e}=e^{-u}e$ for $\tilde{g}$, the respective Strominger curvature components are related by
\begin{eqnarray}
e^{2u}\tilde{R}^s_{k\bar{\ell}i\bar{j}} - R^s_{k\bar{\ell}i\bar{j}} & = & \delta_{i\ell}  \delta_{kj} (-4|u_r|^2) + \delta_{k\ell} (4u_iu_{\overline{j}}) + \delta_{ij} (2u_{k\overline{\ell}} - 4u_r \overline{T^k_{r\ell }}) + \delta_{kj} (-2 u_{i\overline{\ell}})  + \nonumber \\ \label{eq:18}
& & + \ \delta_{i\ell} (-2u_{k\overline{j}} - 4u_{\overline{r}} T^j_{rk}  + 4u_r \overline{T^k_{rj}}) +4u_i \overline{T^k_{j\ell }} + 4u_{\overline{j}} T^{\ell}_{ik}
\end{eqnarray}
for any $1\leq i,j,k,\ell \leq n$, where $r$ is summed up from $1$ to $n$, and subscripts in $u$ stand for covariant derivatives with respect to $\nabla^s$ of $g$.
\end{proposition}

In the above computation, we used the following commutation formula
$$ u_{\overline{j}k} - u_{k\overline{j}} = 2\sum_r \big( u_{\overline{r}} T^j_{rk} - u_r \overline{T^k_{rj}} \big) .$$
At any given point $p\in M$, if we choose the unitary frame $e$ so that $\theta^s|_p=0$, then $\theta|_p=-2\gamma|_p$, and
$$ [e_k, \overline{e}_j] = \nabla^c_{e_k}\overline{e}_j - \nabla^c_{\overline{e}_j} e_k = -2 \gamma_{e_k}\overline{e}_j +2 \gamma_{\overline{e}_j} e_k = 2\sum_r \big( T^j_{rk} \overline{e}_r - \overline{T^k_{rj}} e_r\big).$$
This gives us the above commutation formula, which is independent of the choice of the unitary frame. Taking the symmetrization of (\ref{eq:18}), one has
\begin{eqnarray*}
e^{2u}\widehat{\tilde{R}^s}_{k\overline{\ell}i\overline{j}} - \widehat{R}^s_{k\overline{\ell}i\overline{j}} & = & (\delta_{ij}  \delta_{k\ell}+ \delta_{i\ell}  \delta_{kj}) (-2|u_r|^2) + \delta_{ij} ( -\frac{1}{2}u_{k\overline{\ell}}+ u_ku_{\overline{\ell}} - u_{\overline{r}} T^{\ell}_{rk} ) +\nonumber \\
&& + \ \delta_{kj} ( -\frac{1}{2}u_{i\overline{\ell}}+ u_iu_{\overline{\ell}} - u_{\overline{r}} T^{\ell}_{ri} ) + \delta_{i\ell} ( -\frac{1}{2}u_{k\overline{j}}+ u_ku_{\overline{j}} - u_{\overline{r}} T^{j}_{rk} ) + \nonumber \\
&& + \ \delta_{k\ell } ( -\frac{1}{2}u_{i\overline{j}}+ u_iu_{\overline{j}} - u_{\overline{r}} T^{j}_{ri} )
\end{eqnarray*}
where $r$ is summed up. Following \cite{ADM} and making the substitution $F=e^{-2u}>0$, we get
\begin{eqnarray}
4\widehat{\tilde{R}^s}_{k\overline{\ell}i\overline{j}} - 4F\widehat{R}^s_{k\overline{\ell}i\overline{j}} & = &  -\frac{2}{F}|F_r|^2 (\delta_{ij}  \delta_{k\ell}+ \delta_{i\ell}  \delta_{kj})  + \delta_{ij} ( F_{k\overline{\ell}}+ 2 F_{\overline{r}} T^{\ell}_{rk} ) + \nonumber \\
&& + \  \delta_{kj} ( F_{i\overline{\ell}}+ 2 F_{\overline{r}} T^{\ell}_{ri} ) +  \delta_{i\ell} ( F_{k\overline{j}}+ 2 F_{\overline{r}} T^{j}_{rk} ) +  \delta_{k\ell } ( F_{i\overline{j}}+ 2 F_{\overline{r}} T^{j}_{ri} ) \label{eq:F}
\end{eqnarray}

From this formula, we immediately get the following corollary, which states that on a compact complex space form, any conformal change of the standard metric cannot be a Strominger space form unless it is a scaling (hence K\"ahler).

\begin{proposition} \label{prop4}
Let $(M^n,g)$ be a compact K\"ahler manifold with constant holomorphic sectional curvature, where $n\geq 2$. If $\tilde{g}=\frac{1}{F}g$ is a conformal metric with pointwise constant Strominger holomorphic sectional curvature. Then $F$ is constant.
\end{proposition}

\begin{proof}
Denote by $c$ the holomorphic sectional curvature of $g$ and by $\tilde{c}$ the Strominger holomorphic sectional curvature of $\tilde{g}$. Here $c$ is a constant, and $\tilde{c}$ is a real valued function. Since $g$ is K\"ahler, formula (\ref{eq:F}) gives us
$$ 2\lambda (\delta_{ij}\delta_{k\ell} + \delta_{i\ell}\delta_{kj}) = \delta_{ij}F_{k\overline{\ell}} +  \delta_{k\ell}F_{i\overline{j}} + \delta_{i\ell}F_{k\overline{j}} + \delta_{kj}F_{i\overline{\ell}},$$
where $\lambda = \tilde{c}-Fc + \frac{|F_r|^2}{F}$, and subscripts of $F$ stand for covariant derivatives with respect to the Strominger (which coincide with the Levi-Civita) connection of $g$. Since $n\geq 2$, we deduce
$$ F_{i\overline{j}} =\delta_{ij} \lambda, \ \ \ \ \forall \ 1\leq i, j\leq n. $$
Since $R_{\bar{i}\bar{j}\ast \ast }=0$, we have $F_{i\bar{i}\bar{j}}=F_{i\bar{j}\bar{i}}$. Now for $i\neq j$, we get
$$ \lambda_{\bar{j}}=F_{i\bar{i}\bar{j}}=F_{i\bar{j}\bar{i}} = 0 $$
since $F_{i\bar{j}}=0$. This means that $\lambda$ is a constant. Hence by $\Delta F =n\lambda $ we know that $\lambda =0$ and $F$ is a constant. This completes the proof of the proposition.
\end{proof}

Now we are ready to prove the main result of this article, Theorem \ref{thm1}.

\begin{proof}[\textbf{Proof of Theorem \ref{thm1}}] Let $(M^2,g)$ be a compact Hermitian surface with pointwise constant Strominger holomorphic sectional curvature, namely, $H^s=c$ where $c$ is a smooth real-valued function on $M^2$. By Proposition \ref{prop2}, we know that $(M^2,g)$ is self-dual. Hence by \cite[Theorem 1']{ADM}, we know that $g$ is conformal to a metric $h$ on $M^2$, which is either  (1) a complex space form, or (2) the non-flat, conformally flat K\"ahler metric, or (3) an isosceles Hopf surface. Write $g=\frac{1}{F}h$, where $F$ is a positive smooth function on $M^2$. For case (1), Proposition \ref{prop4} says that $F$ must be constant, hence $(M^2,g)$ itself is a complex space form. The other two cases can be ruled out by similar method as in \cite{ADM}, with minor modifications. For the sake of completeness we give the details below.

In case (2), $M^2$ is a holomorphic ${\mathbb P}^1$-bundle over a curve $C$ of genus at least $2$, and $h$ is locally a product metric where the factors have constant holomorphic sectional curvature $-1$ (for base curve) or $1$ (for the fiber). Following \cite{ADM}, let $(z_1,z_2)$ be local holomorphic coordinates in $M^2$ such that the K\"ahler  form of $h$ is given  by
\begin{equation*}
 \omega_h = 2\sqrt{-1} \frac{dz_1\wedge d\overline{z}_1}{(1-|z_1|^2)^2} + 2\sqrt{-1} \frac{dz_2\wedge d\overline{z}_2}{(1+|z_2|^2)^2} = \omega_1 + \omega_2.
 \end{equation*}
Let $e$ be a local unitary frame of $(M^2,h)$ such that $e_i$ is parallel to $\frac{\partial}{\partial z_i}$ for $i=1$ and $2$. Then under $e$, we have $R^{(h)}_{1\overline{1}1\overline{1}} = -1$ and $R^{(h)}_{2\overline{2}2\overline{2}} = 1$, while the other curvature components vanish. Suppose $g=e^{2u}h = \frac{1}{F}h$ is a Hermitian metric on $M^2$ conformal to $h$ such that the Strominger holomorphic sectional curvature of $g$ is pointwise constant: $H^{s(g)}=c$ where $c$ is a function on $M^2$. Then by Proposition \ref{prop3}, we have
\begin{eqnarray*}
F (R^{s(g)}_{k\overline{\ell}i\overline{j}} - F R^{(h)}_{k\overline{\ell}i\overline{j}})& = & \delta_{i\ell}  \delta_{kj} (-\sum_r|F_r|^2) + \delta_{k\ell} (F_iF_{\overline{j}}) + \delta_{ij} (F_kF_{\overline{\ell}} - FF_{k\overline{\ell}} ) + \\
&& + \ \delta_{kj} (FF_{i\overline{\ell}} -F_iF_{\overline{\ell}})  +  \delta_{i\ell} (FF_{k\overline{j}} - F_kF_{\overline{j}}).
\end{eqnarray*}
Write $\lambda = \frac{1}{F}\sum_r |F_r|^2 + c$. In the above, by letting $(ijk\ell)=(1111)$ or $(2222)$,  or $(1211)$ and $(1112)$ and then add up,  we get
\begin{equation} \label{eq:20}
F_{1\overline{1}}=\lambda +F, \ \ \ \ \ F_{2\overline{2}}=\lambda -F, \ \ \ \ \ F_{1\overline{2}}=0,
\end{equation}
where subscripts stand for covariant derivatives in $(M^2,h)$ under the frame $e$.  We want to deduce a contradiction from these equations. Write $x=\lambda +F$ and $y=\lambda -F$. We have
$x_{\overline{2}} = F_{1\bar{1}\bar{2}}= F_{1\bar{2}\bar{1}} =0$. So $x$ is a function of $z_1$ and $\overline{z}_1$ along. Similarly, $y_{\overline{1}}=0$ and $y$ is a function of $z_2$ and $\overline{z}_2$ along. Since $x-y=2F$, we have
$$ \Delta x = \Delta_1 x = \Delta_1 (y+2F) = 2\Delta_1 F=2F_{1\overline{1}} = 2x.$$
Similarly, $\Delta y=-2y$. Since $M$ is compact, at the maximum point of $x$, we have $x_{max}=2\Delta x \leq 0$, hence $x\leq 0$. Similarly, $x_{min} = 2\Delta x \geq 0$ hence $x\geq 0$. This will force $x=0$ identically. Hence $\lambda =-F$,  $F_{1\overline{1}}=0$, and $F_{2\overline{2}}=-2F$. That is, $\Delta F = -2F$. Let $p$ be a point where $F$ reaches its minimum value. At this point, $\Delta F \geq 0$, hence $F_{min} = -\frac{1}{2}\Delta F \leq 0$. This of course will contradict with the fact that $F>0$ everywhere.

For case (3), when $(M^2,h)$ is an isosceles Hopf surface with $h$ the standard Hopf metric, and $g$ is a metric conformal to $h$ and having $H^s=f$ where $f$ is a function. Then by Proposition \ref{prop1}, we know that $g$ must be a constant multiple of one of the metric $g_A$, hence is an admissible metric. This completes the proof of Theorem \ref{thm1}.
\end{proof}

\vspace{0.3cm}

\section{In higher dimensions}

Now let us turn our attention to high dimensions. We want to explore the set of all non-K\"ahler Strominger (weak) space forms. At this point we know little about them except isosceles Hopf manifolds and Strominger flat manifolds. In this section, we will prove two non-existence theorems mimicking  \cite{LZ} and \cite{RZ} in the Chern connection case.

Let us begin with the class of {\em complex nilmanifolds} $(M^n,g)$. This means that the universal cover is a nilpotent Lie group $G$  equipped with a left-invariant complex structure $J$ and a compatible left-invariant metric $g$, and $(G,J,g)$ is the universal cover of $(M^n,g)$ as a Hermitian manifold.  Assume that $g$ has pointwise constant Strominger holomorphic sectional curvature, namely $H^s=f$ where $f$ is a function on $M$. We want to draw the conclusion that $G$ is abelian and $g$ is K\"ahler and flat, thus establishing the proof of Theorem \ref{thm3}.

We will follow the notations of \cite{VYZ}, \cite{ZZ1}, and \cite{LZ}. Let ${\mathfrak g}$ be the Lie algebra of $G$. Left-invariant metric or complex structure on $G$ corresponds to inner product or complex structure on on ${\mathfrak g}$. The latter means linear transformation $J$ on ${\mathfrak g}$ satisfying $J^2=-I$ and the integrability condition
\begin{equation*}
[x,y]- [Jx,Jy] +J [Jx, y] + J[x,Jy] =0, \ \ \ \ \forall \ x,y \in {\mathfrak g}.
\end{equation*}
We will extend $J$ and $g=\langle \,,\,\rangle$ linearly over ${\mathbb C}$ to the complexification ${\mathfrak g}^{\mathbb C}$  of ${\mathfrak g}$, and we have the decomposition ${\mathfrak g}^{\mathbb C} = {\mathfrak g}' \oplus {\mathfrak g}''$ into the $(1,0)$ and $(0,1)$ parts, where ${\mathfrak g}'' = \overline{{\mathfrak g}'}$, and
$${\mathfrak g}'= \{ x-\sqrt{-1}Jx \mid x\in {\mathfrak g}\} . $$
Let $e=\{ e_1, \ldots , e_n\}$ be a unitary basis of ${\mathfrak g}'$. We will follow the notations in \cite{VYZ} (see also \cite{YZ1} and \cite{ZZ1}), and denote by
\begin{equation} \label{CandD}
C_{ik}^j = \langle [e_i,e_k] , \,\overline{e_j} \rangle , \ \ \ \ \ D_{ik}^j = \langle  [\overline{e_j}, e_k] , e_i \rangle
\end{equation}
for any $1\leq i,j,k\leq n$. By the integrability of $J$ we have
\begin{equation*}
[e_i,e_k] = \sum_{j=1}^n C_{ik}^j e_j , \ \ \ \ \ [\overline{e_j}, e_k] = \sum_{i=1}^n \big( D_{ik}^j \overline{e_i} - \overline{D^k_{ij}} e_i  \big).
\end{equation*}
Denote by $\varphi$ the coframe dual to $e$, then we have the structure equation
\begin{equation*}
d\varphi_i = -\sum_{j,k=1}^n \big( \frac{1}{2} C^i_{jk}\varphi_j\wedge \varphi_k  + \overline{D^j_{ik}} \varphi_j \wedge \overline{\varphi}_k \big)
\end{equation*}
and the first Bianchi identity which now takes the form
\begin{eqnarray*}
& & \sum_{r=1}^n \big( C^r_{ij}C^{\ell}_{rk} + C^r_{jk}C^{\ell}_{ri} + C^r_{ki}C^{\ell}_{rj} \big) \ = \ 0  \label{CC} \\
& & \sum_{r=1}^n \big( C^r_{ik}D^{\ell}_{jr} + D^r_{ji}D^{\ell}_{rk} - D^r_{jk}D^{\ell}_{ri} \big) \ = \ 0  \label{CD} \\
& & \sum_{r=1}^n \big( C^r_{ik}\overline{D^{r}_{j\ell }} - C^j_{rk}\overline{D^{i}_{r\ell }} + C^j_{ri}\overline{D^{k}_{r\ell }}  - D^{\ell}_{ri}\overline{D^{k}_{jr }} + D^{\ell}_{rk} \overline{ D^{i}_{jr }}  \big) \ = \ 0  \label{CDbar}
\end{eqnarray*}
for any $1\leq i,j,k,\ell \leq n$. Under the frame $e$, the Chern connection form and torsion components are
\begin{equation*}
\theta_{ij}=\sum_{k=1}^n \big( D^j_{ik}\varphi_k  -\overline{D^i_{jk}} \overline{\varphi_k} \big), \ \ \ \     2T_{ik}^j = - C_{ik}^j - D_{ik}^j + D_{ki}^j.   \label{T}
\end{equation*}
From this, we get the expression for $\gamma$ and the Strominger connection matrix
\begin{equation*}
\theta^s_{ij}=\sum_{k=1}^n \{ (-C^j_{ik} + D^j_{ki}) \varphi_k  + ((\overline{ C^i_{jk}} - \overline{D^i_{kj}} ) \, \overline{\varphi}_k \} .
\end{equation*}
Taking the $(1,1)$ part in $\Theta^s=d\theta^s - \theta^s  \wedge \theta^s $, we get
\begin{eqnarray}
 R^s_{k\overline{\ell}i\overline{j}} & = & \big(C^r_{ik}\overline{C^r_{j\ell } } - C^j_{rk}\overline{C^i_{r\ell } }\big) - \big( C^r_{ik} \overline{D^r_{\ell j} } + \overline{C^r_{j\ell} } D^r_{ki}  \big) + \big( C^j_{ir} \overline{D^k_{r\ell } } - C^j_{kr} \overline{ D^i_{\ell r} } \big) + \nonumber \\
&& + \ \big( \overline{ C^i_{jr}} D^{\ell}_{rk }  - \overline{C^i_{\ell r} } D^{j}_{kr}  \big) - \big( D^j_{ri} \overline{ D^k_{r\ell }} + D^{\ell}_{rk}  \overline{D^i_{rj} }  \big) + \big( D^r_{ki} \overline{ D^r_{\ell j}} - D^{j}_{kr}   \overline{D^i_{\ell r} }  \big) \label{LieRs}
\end{eqnarray}
for any $i,j,k,\ell$. Here $r$ is summed up from $1$ to $n$. Taking the symmetrization, we get the following:

\begin{proposition} \label{LieRshat}
Let $(G,J,g)$ be an even-dimensional Lie group equipped with a left-invariant complex structure and a compatible metric. Let $e$ be a unitary frame of ${\mathfrak g}'$ and $C$, $D$ be defined by (\ref{CandD}). Then under $e$ the ($(1,1)$-part of) Strominger curvature components are given by (\ref{LieRs}) and their symmetrization are given by
\begin{eqnarray}
 4\widehat{R}^s_{k\overline{\ell}i\overline{j}} & = & - \big( C^j_{rk}\overline{C^i_{r\ell } } + C^j_{ri}\overline{C^k_{r\ell } } + C^{\ell}_{rk}\overline{C^i_{rj } } + C^{\ell}_{ri}\overline{C^k_{rj } } \big)
 + \big( C^j_{ir}\overline{D^k_{r\ell } } + C^j_{kr}\overline{D^i_{r\ell } } + C^{\ell}_{ir}\overline{D^k_{rj } } + C^{\ell}_{kr}\overline{D^i_{rj } } \big)  \nonumber \\
 && - \, \big( C^j_{ir}\overline{D^k_{\ell r} } + C^j_{kr}\overline{D^i_{\ell r} } + C^{\ell}_{ir}\overline{D^k_{jr } } + C^{\ell}_{kr}\overline{D^i_{jr } } \big)
 + \big( \overline{C^i_{jr}} D^{\ell}_{rk} + \overline{C^k_{jr}} D^{\ell}_{ri} + \overline{C^i_{\ell r}} D^{j}_{rk} + \overline{C^k_{\ell r}} D^{j}_{ri} \big) \nonumber \\
 && - \, \big( \overline{C^i_{jr}} D^{\ell}_{kr} + \overline{C^k_{jr}} D^{\ell}_{ir} + \overline{C^i_{\ell r}} D^{j}_{kr} + \overline{C^k_{\ell r}} D^{j}_{ir} \big)
 -2 \big( D^j_{ri}\overline{D^k_{r\ell } } + D^j_{rk}\overline{D^i_{r\ell } } + D^{\ell}_{ri}\overline{D^k_{rj } } + D^{\ell}_{rk}\overline{D^i_{rj } } \big)  \label{LieRshat} \\
 &&  + \,\big( D^r_{ki}\overline{D^r_{\ell j} } + D^r_{ik}\overline{D^r_{\ell j} } + D^{r}_{ki}\overline{D^r_{j\ell } } + D^{r}_{ik}\overline{D^r_{j\ell } } \big)   - \, \big( D^j_{ir}\overline{D^k_{\ell r} } + D^j_{kr}\overline{D^i_{\ell r} } + D^{\ell}_{ir}\overline{D^k_{jr } } + D^{\ell}_{kr}\overline{D^i_{jr } } \big) \nonumber
\end{eqnarray}
for any $1\leq i,j,k,\ell \leq n$, where $r$ is summed up.
\end{proposition}

In particular, letting $i=j$ and $k=\ell$, we have
\begin{eqnarray}
 4\widehat{R}^s_{k\overline{k}i\overline{i}} & = & - \,\big( |C^i_{rk}|^2 + |C^k_{ri}|^2 + 2\mbox{Re} \{ C^{i}_{ri}\overline{C^k_{rk } }\}  \big)  + 2\mbox{Re} \{ \overline{C^i_{ir}} (D^k_{rk}-D^k_{kr}) + \nonumber  \\
 && +\, \overline{C^k_{kr}} (D^i_{ri}-D^i_{ir}) + \overline{C^i_{kr}} (D^k_{ri}-D^k_{ir}) + \overline{C^k_{ir}} (D^i_{rk}-D^i_{kr}) \}+ \label{LieRshatik}  \\
 && -\,   2\big( |D^i_{rk}|^2 + |D^k_{ri}|^2 + 2\mbox{Re} \{ D^{i}_{ri}\overline{D^k_{rk } } \} \big) + \big( |D^r_{ki}|^2 + |D^r_{ik}|^2 + 2\mbox{Re} \{ D^{r}_{ik}\overline{D^r_{ki } } \} \big) \nonumber \\
 && - \, \big( |D^i_{kr}|^2 + |D^k_{ir}|^2 + 2\mbox{Re} \{ D^{i}_{ir}\overline{D^k_{kr } } \} \big) \nonumber \\
  \widehat{R}^s_{i\overline{i}i\overline{i}} & = &    - |C^i_{ri}|^2  + 2\mbox{Re} \{ \overline{C^i_{ir}} (D^i_{ri}-D^i_{ir})\} -2  |D^i_{ri}|^2 + |D^r_{ii}|^2 - |D^i_{ir}|^2 \label{LieRshatii}
\end{eqnarray}

Note that so far we have  not used the assumption that $G$ is nilpotent yet. For nilpotent groups, a famous result of Salamon \cite[Theorem 1.3]{Salamon} states the following:

\begin{theorem} [Salamon] Let $G$ be a nilpotent Lie group of dimension $2n$ equipped with a left invariant complex structure. Then there exists a coframe $\varphi =\{ \varphi_1, \ldots , \varphi_n\}$ of left invariant $(1,0)$-forms on $G$ such that $$ d\varphi_1 =0, \ \ \ d\varphi_i = {\mathcal I} \{\varphi_1, \ldots , \varphi_{i-1}\} , \ \ \ \forall \ 2\leq i\leq n, $$
where ${\mathcal I}$ stands for the ideal in  exterior algebra of the complexified cotangent bundle generated by those $(1,0)$-forms.
\end{theorem}

When $g$ is a compatible left-invariant metric, clearly one can choose the above coframe $\varphi$ so that it is also unitary. In terms of the structure constants $C$ and $D$ given by (\ref{CandD}), this means
\begin{equation}
C^j_{ik}=0  \ \ \ \mbox{unless} \ \ j>i \ \mbox{or} \ j>k; \ \ \ \ \ D^j_{ik}=0  \ \ \ \mbox{unless} \ \ i>j.  \label{Salamon}
\end{equation}

If the complex structure $J$ is nilpotent in the sense of  Cordero, Fern\'{a}ndez, Gray and Ugarte \cite{CFGU}, then there exists invariant unitary coframe $\varphi$ so that
\begin{equation}
C^j_{ik}=0  \ \ \ \mbox{unless} \ \ j>i \ \mbox{and} \ j>k; \ \ \ \ \ D^j_{ik}=0  \ \ \ \mbox{unless} \ \ i>j \ \mbox{and} \ i>k.  \label{CFGU}
\end{equation}

While it is an interesting question to know whether there are non-K\"ahler Strominger space forms amongst complex nilmanifolds, namely, does (\ref{Salamon}) and $H^s=c$ imply $C=D=0$? We do not know the answer of that at this point, so we settle for the easier case, namely, proving that (\ref{CFGU}) plus $H^s=c$ does imply $C=D=0$.

\begin{proof}[\textbf{Proof of Theorem \ref{thm3}}]
Let $(G,J,g)$ be a nilpotent Lie group equipped with a left-invariant complex structure and compatible left-invariant metric $g$. Assume that the complex structure $J$ is nilpotent in the sense of Cordero, Fern\'{a}ndez, Gray and Ugarte in \cite{CFGU}. We have unitary coframe $\varphi$ such that $C$ and $D$ satisfy (\ref{CFGU}).

Assuming that the Strominger holomorphic sectional curvature is pointwise constant, namely, $H^s=f$ for some function $f$. This is equivalent to
$$ \widehat{R}^s_{k\overline{\ell }i\overline{j}} = \frac{f}{2}\big( \delta_{ij} \delta_{k\ell} + \delta_{i\ell } \delta_{kj} \big) $$
under any unitary frame. In particular, $\widehat{R}^s_{k\overline{k }i\overline{i}} = \frac{f}{2}\big( 1 + \delta_{ik} \big)$. Utilizing (\ref{CFGU}), the formula (\ref{LieRshatii}) now gives us
$$ f =  -2\sum_{r>i} |D^i_{ri}|^2 $$
for each $1\leq i\leq n$. Letting $i=n$, we see that $f\equiv 0$, so the above equality gives us $D^i_{\ast i}=0$ for each $i$. Similarly, for any $i<k$, by utilizing (\ref{CFGU}) and (\ref{LieRshatik}), together with the fact that $f=0$ and $D^j_{\ast j}=0$ for any $j$, we get
\begin{equation}
\label{eq:nilD}
\sum_{r<k} |C^k_{ri}-D^i_{kr}|^2 + 2\sum_{r>k} \big( |D^i_{rk}|^2 + |D^k_{ri}|^2  \big)\ =\ \sum_{r<k}|D^r_{ki}|^2, \ \ \ \ \ \forall \ i<k.
\end{equation}
Let $k=2$. Then the right hand side is zero, hence the left hand side must also be zero which leads to $D^1_{\ast 2}=D^2_{\ast 1}=0$. Therefore we have $D^{\ast}_{3\ast}=0$. Now let $k=3$ in (\ref{eq:nilD}), the right hand side is zero, so the left hand side must be zero which implies $D^{\ast}_{4\ast}=0$. Repeat this process, we know eventually that $D=0$. By (\ref{eq:nilD}) again, we get $C=0$.  So $G$ is the abelian group, and $g$ is K\"ahler and flat. This completes the proof of Theorem \ref{thm3}.
\end{proof}

\begin{proof}[\textbf{Proof of Theorem \ref{thm4}}]
Let $(M^n,g)$ be a Strominger K\"ahler-like (SKL) manifold with pointwise constant Strominger holomorphic sectional curvature. If $g$ is not K\"ahler, then by \cite{YZZ} we know that locally there always exist admissible frames, which are unitary frames with $\nabla^s e_n=0$, and $R^s_{i\overline{j}k\overline{n}}=0$ for any $1\leq i,j,k\leq n$. The SKL condition implies that $\widehat{R}^s_{i\overline{j}k\overline{\ell}} = R^s_{i\overline{j}k\overline{\ell}}$, and the pointwise constant Strominger holomorphic sectional curvature assumption means that there is a function $f$ on $M$ such that $\widehat{R}^s_{i\overline{j}k\overline{\ell}} = \frac{f}{2}(\delta_{ij}\delta_{k\ell} + \delta_{i\ell}\delta_{kj})$. So by $R^s_{n\overline{n}n\overline{n}}=0$, we get $f\equiv 0$, hence $R^s= \widehat{R}^s =0$, that is, $g$ is Strominger flat. This completes the proof of Theorem \ref{thm4}.
\end{proof}

\vs

\noindent\textbf{Acknowledgments.} The second named author would like to thank Haojie Chen, Xiaolan Nie, Kai Tang, Bo Yang, and Quanting Zhao for helpful discussions.

\vs

\end{document}